\documentclass[a4paper,12pt]{amsart}
\usepackage{amsmath,amsthm}
\usepackage{amsfonts,amsmath,amsthm, amssymb, mathrsfs}
\usepackage{float}
\usepackage{euscript,verbatim}
\usepackage{graphicx}
\usepackage[usenames]{color}
\usepackage[colorlinks,linkcolor=red,anchorcolor=blue,citecolor=blue]{hyperref}
\usepackage{amsmath}
\usepackage{amsthm}
\usepackage[all]{xy}
\usepackage{graphicx} 
\usepackage{amssymb} 
\usepackage{enumerate}
\usepackage{lipsum}

\newtheorem{thm}{Theorem}[section]
\newtheorem{prop}[thm]{Proposition}
\newtheorem{df}[thm]{Definition}
\newtheorem{lem}[thm]{Lemma}
\newtheorem{cor}[thm]{Corollary}

\newtheorem{ex}[thm]{Example}

\newtheorem{rem}[thm]{Remark}

\def\N{\mathbb{N}}

\def\R{\mathbb{R}}

\def\M{\mathcal{M}_{\mathbb{P}}(\Omega \times X)}

\def\N{\mathbb{N}}

\def\N{\mathbb{N}}
\def\R{\mathbb{R}}
\def\C{\mathcal{C}}
\def\S {S_{n\tau}\varphi({\omega_{\C_{[0,n)}(x)}})}
\def\M {S_{n_i\tau}\varphi({\omega_{\C_{[0,n_i)}(x_i)}})}

\makeatletter 
\@addtoreset{equation}{section}
\numberwithin{equation}{section}

\begin{document}

\title{Bowen's  equations for invariance pressure of control systems}

\author{Rui Yang$^{1,2,3}$, Ercai Chen$^{1}$, Jiao Yang$^{1}$ and Xiaoyao Zhou$^{1}$*
}
\address
{1.School of Mathematical Sciences and Institute of Mathematics, Ministry of Education Key Laboratory of
	NSLSCS, Nanjing Normal University, Nanjing 210023, Jiangsu, P.R.China}

\address
{2.College of Mathematics and Statistics, Chongqing University, Chongqing 401331, P.R.China}
\address
{3. Key Laboratory of Nonlinear Analysis and its Applications (Chongqing University), 
Ministry of Education}

\email{zkyangrui2015@163.com}
\email{ecchen@njnu.edu.cn}
\email{jiaoyang6667@126.com} 
\email{zhouxiaoyaodeyouxian@126.com}
\date{}

\begin{abstract}
 We aim to establish Bowen's equations for upper capacity invariance pressure   and  Pesin-Pitskel invariance pressure of discrete-time control systems. We first  introduce  a new invariance pressure called induced invariance pressure  on partitions that  specializes the  upper capacity invariance pressure on partitions,  and  then   show  that   the  two types of invariance pressures  are related by a Bowen's equation.  Besides,  to establish Bowen's equation for  Pesin-Pitskel invariance pressure on partitions  we also  introduce a new notion called BS invariance dimension on subsets.  Moreover,  a variational principle for BS  invariance dimension on subsets is established.
\end{abstract}
\keywords{Control systems; Bowen's equations; Invariance pressure-like quantities; BS invariance dimension; Variational principle;}
\subjclass[2000]{37A35, 37B40, 93C10, 93C55}
\renewcommand{\thefootnote}{}
\footnote{*corresponding author}
\renewcommand{\thefootnote}{\arabic{footnote}}
\maketitle
\section{Introduction}
Entropy  is  well-known as an   important  quantity  for   characterizing the topological complexity of dynamical systems. In the context of classical dynamical systems,  Adler,  Konheim  and   McAndrew   \cite{akm65}  introduced  topological entropy by  using open covers.    Bowen \cite{b71}    formulated the equivalent  definitions for topological entropy  by using spanning sets and separated sets.  In 1973,  Bowen  \cite{b73} defined the topological entropy on subsets  resembling  the definition of Hausdorff dimension, which is called Bowen topological entropy later.   By modifying    Carath\'eodory's construction,  which  we call  Carath\'eodory-Pesin  structures,  Pesin   in  his monograph axiomatically \cite{p97} introduced the notions of    dimensions and   capacities on subsets without involving dynamics. Based on  this structure, Pesin and Pitskel \cite{pp84}  extended Bowen topological entropy to  topological pressure  on subsets, and established variational principle for  some certain subsets in terms of measure-theoretic entropy. Employing a geometric measure theory approach,    Feng and Huang  \cite{fh12} formulated  variational principle for Bowen  topological   entropy of compact subset in terms of measure-theoretic  lower  Brin-Katok local entropy of Borel probability measure \cite{bk83}.   An extension of   the result,  presented by   Wang and Chen  \cite{wc12},  shows  Feng-Huang's variational principle is still valid for   BS dimension \cite{bs00}.

Parallel to the classical entropy theory,     in the framework of control  systems, entropy   is  more  closely  associated with  the realization of   certain control  tasks.  Analogous to the  definition of topological entropy,   Nair et al. \cite{NM04}  firstly  introduced  topological feedback entropy by using  open covers,  which  gives  a  characterization for  the minimal data  rate of certain control tasks. Besides,  Colonius and Kawan \cite{ck} gave the notion of invariance  entropy   to  measure the exponential growth rate of the minimal number of control  functions necessary to keep a subset of  a  controlled invariant set invariant.   It turns out that   the  two different  invariance entropies coincide up to the technical assumptions \cite{ckn13}.   As a natural generalization of topological entropy,  topological pressure  \cite{w82}  plays a vital role  in thermodynamic formalism, and  is a powerful tool  for studying  dimension theory  as well as other  fields of dynamical systems.  In 2019,  Colonius et al. \cite{fsc19} defined   invariance pressure for discrete control systems. It  measures  the exponential growth rate of  the  total weighted information produced by the control functions  such that  the  trajectories  of systems remain in a controlled invariant set.  The work in   \cite{hz18,zh19} suggests  that invariance  pressure-like quantities  also  enjoy  the   dimensional characterizations by invoking the  Carath\'eodory-Pesin  structures.  It is reasonable to ask how to  inject ergodic theoretic ideas into  invariance entropy theory of control systems by establishing some proper   variational principles for invariance  pressure-like quantities.    Colonius  et al.  \cite{ff18,f18} introduced metric invariance entropy with respect to conditionally invariant measures and obtained partial variational principle.  Later,  inspired by  the  Feng-Huang's work \cite{fh12},  Wang, Huang and Sun \cite{whs19}  derived  an analogous   variational principle for Bowen   invariance entropy  on  compact subsets. See also   \cite{z20,nwh22,wh22}  for more advances about the variational principles   of invariance pressure-like quantities.

The present  paper aims to  introduce some  new invariance pressure-like quantities,  and  establish the bridges between the new  quantities  and those old ones given in \cite{fsc19,whs19,zh19,z20}  by some Bowen's equations. Before that,  mentioning some  backgrounds of  the well-known Bowen's equations in classical dynamical systems  is useful for stating our main results.   In 1979, Bowen \cite{b79}   found  that the Hausdorff dimension of a quasi-circle is given by the unique zero of the pressure function of a geometric potential. This means that the Hausdorff dimension of  certain compact  sets  can be computed by figuring out  the root of the corresponding equation, rather than by  its  definition. Since then,  the equation involving the unique zero of certain pressure function was  customarily called Bowen's equation  in a unified  way.  In 2000,   Barreira and Schmeling  \cite{bs00}  proved that   BS dimension is  the unique root of the equation defined by topological pressure  of additive potential function.  By   extending the concept of  induced topological pressure, introduced by  Jaerisch et al. \cite{jms14} in symbolic systems,  to  general topological dynamical systems,  Xing  and Chen \cite{xc}    showed that   the induced topological pressure is exactly the unique root  of the   equation defined by   classical topological pressure.  More results of Bowen's equations and dimensions  can be  found in  \cite{c11,ycz22,xm23, bm221}. See also  Ruelle \cite{rue82}, Mihailescu  et. al. \cite{mm83,mu04, mu10,ms13} and the references
therein  for the  applications  of  Bowen's equations   in  dimension theory.   The aforementioned work  not only suggests  that Bowen's  equations provide  a bridge between  thermodynamic formalism and  dimension theory of dynamical systems,   but also exhibits  an  approximate estimate  for  the geometric dimension of the certain sets.   This motivates us to  consider   the so-called   Bowen's equations  for  invariance pressure-like quantities of  control systems. 

It  is  difficult to give an equivalent definition  for invariance pressure in control systems   using separated  sets as it was done for topological pressure. So  we define some new invariance pressures of  controlled  invariant sets on  invariant partitions  so that the Bowen's equation  is available.   Inspired by  the  ideas of \cite{ jms14, xc,ycz22},  we   first define   induced  invariance pressure on partitions by spanning sets and separated sets, and then prove that the induced  invariance pressure   on partitions  can be  characterized by a certain  dimensional  structure and  that it is the unique root of the equation  defined by upper capacity invariance pressure on partitions  introduced  by Zhong and Huang   \cite{zh19}.    

\begin{thm}\label{thm 1.1.1}
Let  $\varSigma=(\mathbb{N},X,U,\mathcal{U},\phi)$ be a discrete-time control system. Let $Q$ be a  controlled  invariant set and  $\mathcal C =(\mathcal A, \tau,\nu)$ be an invariant partition of  $Q$, and let $\varphi,\psi \in C(U,\mathbb{R})$ with $\psi>0$. 
Then
\begin{align*}
P_{inv,\psi }(\varphi,Q,\C)=\inf\{\beta \in \mathbb{R}: \limsup_{T \to \infty}R_{inv,\psi, T}(\varphi-\beta\psi,Q,\C)<\infty\},
\end{align*}
where $P_{inv,\psi }(\varphi,Q,\C)$  denotes   $\psi$-induced  invariance pressure  of $\varphi$ on $Q$ w.r.t. $\mathcal{C}$.
\end{thm}

\begin{thm}\label{thm 1.1}
Let  $\varSigma=(\mathbb{N},X,U,\mathcal{U},\phi)$ be a discrete-time control system.  Let $Q$ be a  controlled  invariant set and  $\mathcal C =(\mathcal A, \tau,\nu)$ be an invariant partition of  $Q$, and let $\varphi,\psi \in C(U,\mathbb{R})$ with $\psi>0$. Then  $P_{inv,\psi }(\varphi,Q,\C)$ is the unique root  of  the equation $$P_{inv}(\varphi-\beta \psi,Q,\mathcal{C})=0,$$
where $P_{inv}(\varphi, Q,\mathcal{C})$ denotes  the upper capacity invariance pressure of $\varphi$ on $Q$ w.r.t. $\mathcal{C}$. 

\end{thm}

Since invariance pressure can be  treated as dimension \cite{zh19,z20},   this motivates us  to  establish the Bowen's  equation for Pesin-Pitskel invariance pressure. To this end,  a new    dimension,  called BS invariance dimension defined by  Carath\'eodory-Pesin  structures,  is introduced  to obtain the precise root  of the  equation defined by  Pesin-Pitskel invariance pressure.   Moreover,  we establish an analogous Feng-Huang's variational principle for BS  invariance dimension  on subsets of partitions to link measure theory and invariance entropy  theory of control systems.   

\begin{thm}\label{thm 1.2} 
Let  $\varSigma=(\mathbb{N},X,U,\mathcal{U},\phi)$ be a discrete-time control system. Let $Q$ be a   controlled  invariant set and  $\mathcal C =(\mathcal A, \tau,\nu)$ be an  invariant partition of  $Q$, and let  $Z$ be a  non-empty subset of $Q$ and $\varphi\in C(U,\mathbb{R})$ with $\varphi >0$. Then  ${\rm dim}_{\C}^{BS}(\varphi,Z,Q)$ is the unique root of  the  equation $$P_{\C}(-t\varphi,Z,Q)=0,$$  where  $P_{\C}(\varphi,Z,Q)$ denotes Pesin-Pitskel invariance pressure  of $\varphi$ on $Z$ w.r.t. $\C$, and ${\rm dim}_{\C}^{BS}(\varphi,Z,Q)$ is  BS invariance  dimension   of $\varphi$ on $Z$ w.r.t. $\C$.
\end{thm}

One says that a   finite Borel partition  $\mathcal{A}$ of  the  metric space $Q$ is  \emph{clopen} if each element of $\mathcal{A}$ is   both an  open and  a closed Borel  subset of $Q$.

\begin{thm}\label{thm 1.3}
Let  $\varSigma=(\mathbb{N},X,U,\mathcal{U},\phi)$ be a discrete-time control system. Let $Q$ be a  controlled  invariant set and  $\mathcal C =(\mathcal A, \tau,\nu)$ be a clopen  invariant partition of  $Q$. Suppose that $K$ is a  non-empty compact set of $Q$ and $\varphi\in C(U,\mathbb{R})$ with $\varphi >0$.   Then
\begin{align*}
{\rm dim}_{\C}^{BS}(\varphi,K,Q)=\sup\{\underline{h}_{\mu,inv}(\varphi,Q,\C):\mu \in M(Q), \mu (K)=1\},
\end{align*}
where   $\underline{h}_{\mu,inv}(\varphi,Q,\C)$ is  the measure-theoretic  lower  BS  invariance pressure of $\mu$  w.r.t. $\C$ and $\varphi$. 
\end{thm}
  
The rest of this paper is organized as follows. In  section 2,  we   recall some basic settings and fundamental concepts of discrete control systems.  In section 3,  we  introduce the notions of  induced invariance pressure and upper capacity invariance pressure on partitions  by spanning sets and separated sets,  and   prove Theorems \ref{thm 1.1.1} and \ref{thm 1.1}.  In section 4,   we  introduce the notion of BS invariance dimension  and  prove Theorems \ref{thm 1.2} and \ref{thm 1.3}.  

\section{The setup of discrete-time  control systems} 
In this section,  
we    briefly  review   some   concepts associated with entropy  within discrete-time control systems. A systematic treatment of  the  invariance entropy theory  of  deterministic control systems is due to  Kawan's monograph \cite{k13}. Moreover,  the comparisons  between invariance entropy and classical topological entropy are  also presented to clarify the differences between them.

Throughout this paper, we  focus on  a discrete-time  control system  on a  metric space $X$   of the following form
\begin{align}\label{2.1}
x_{n+1}=F(x_n,u_n)=F_{u_n}(x_n), n\in \mathbb{N}=\{0,1,...\},
\end{align}
where  the control-value space $U$ is a compact  metric space, and  $F$ is a  map from $X\times U \to X$  such that   $F_u(\cdot):=F(\cdot,u)$ is  continuous  on $X$ for each  fixed $u\in U$.  Given a  control sequence $\omega=(\omega_0,\omega_1,...)$ in $U$,  the  solution of $(2.1)$ can be  written as
$$\phi(k,x,\omega)=F_{\omega_{k-1}}\circ\cdots\circ F_{\omega_0}(x).$$
For convenience, by  the quintuple $\varSigma:=(\mathbb{N},X,U,\mathcal{U},\phi)$  we denote  the above control system, where $\mathcal{U}$ is a subset of  all sequences $\omega=(\omega_k)_{k \in \mathbb{N}}$ of elements in the control range $U$.

We present  some examples of discrete-time  control systems. Readers can turn  to \cite{k13,hz18,wyc19} for more interesting examples of this aspect.

\begin{ex}
A standard model of (\ref{2.1}) is  the following  scalar linear system of  form 
\begin{align*}
x_{k+1}=qx_k+u_k,
\end{align*}
where $q\in (0,1)$ is a constant, $x_k\in \mathbb{R}_{+}$, and  $u_k $  takes value  in a  compact subset $U$ of $\mathbb{R}_{+}$.  This  gives  a  control system $\varSigma=(\mathbb{N},\mathbb{R}_{+},U,U^{\mathbb{N}},\phi)$,  where the point  $x\in \mathbb{R}_{+}$ 
under the control function $\omega=(u_0,u_1,...)$ at  time $n$ is 
$$\phi(n,x,\omega)=q^nx+\sum_{k=0}^{n-1}q^{n-1-k}u_k.$$
\end{ex}

\begin{ex}
Let  $\{a,b,c\}$ be a discrete topology space, and let  the  state space $X=\{a,b,c\}^{\mathbb{N}}$ be the  product space endowed with product topology. The   map   $\sigma$  on $X$ is  the usual left shift given by  $\sigma(x)=(x_n)_{n+1}$ for each $x=(x_n)_{n\geq 1}$.   
This yields  a discrete-time control system  $\varSigma=(\mathbb{N},X,U,U^{\mathbb{N}},\phi)$ of (\ref{2.1}),  where  $U=\{\sigma,\sigma^2,\sigma^3\}$ is the control-value space. 
\end{ex}

To clarify the essential differences  and  the similarities in the entropy theory between control systems and classical topological  dynamics,    we need to mention  the   definition of topological entropy  and  make a comparison with invariance entropy. 

 Notice that each control   function $\omega=(\omega_0, \omega_1,...)$ induces a discrete non-autonomous dynamical system  $(X,\{F_{\omega_k}\}_{k=0}^{\infty})$, where   $\{F_{\omega_k}\}_{k=0}^{\infty}$  is     a family of  continuous self-maps    on $X$ given in  the expression (\ref{2.1}).  Specially, when   $\omega=(u,u,...)$   is a constant control sequence,   by a pair $(X,T)$ we mean a  discrete topological dynamical system, where $T:=F_u$ is the continuous map from $X$ to $X$. 
 
Define a family  of  Bowen metrics  on $X$ as $$d_n^{\omega}(x,y)=\max_{0\leq j\leq n-1}d (F_1^jx, F_1^jy),$$ where $x,y \in X$, $n\in \mathbb{N}$,  $F_1^0=id$, and $F_1^n=F_{\omega_{n-1}}\circ\cdots\circ F_{\omega_{1}}\circ F_{\omega_{0}}$, for $n\geq 1$,  is the  composition of   the maps $F_{\omega_{j}},j=0,...,n-1$. Then  \emph{the Bowen open  ball and closed ball} of $x$  with radius $\epsilon$  in the metric $d_n^{\omega}$  are  given by 
$$B_n^{\omega}(x,\epsilon)=\{y\in X: d_n^{\omega}(x,y)<\epsilon\},$$
$$\overline B_n^{\omega}(x,\epsilon)=\{y\in X:d_n^{\omega}(x,y)\leq\epsilon\},$$
respectively.

 Given  a non-empty  subset $K$ of $X$, $\epsilon >0$ and $n\in \mathbb{N}$,  a set $E \subset X$ is a  $(n,\omega, \epsilon)$-\emph{spanning set}  of $K$ if   for any $x\in K$, there exists $y \in  E$ such that  $d_n^{\omega}(x,y)\leq \epsilon$.   The smallest cardinality of $(n,\omega,\epsilon)$-spanning  sets  of $K$ is denoted by $r(n,\omega,\epsilon,K)$\footnote{Since  we  only require that  $X$ is a metric space, not necessarily compact,   so $r(n,\omega,\epsilon,K)$ may be $\infty$.}. A set $F \subset K$ is a  $(n,\omega,\epsilon)$-\emph{separated set}  of $K$ if   $d_n^{\omega}(x,y)> \epsilon$   for any  distinct $x,y \in F$.   The largest cardinality of $(n,\omega,\epsilon)$-separated  sets  of $K$ is denoted by $s(n,\omega,\epsilon,K)$. The  \emph{topological entropy} of   $\{F_{\omega_k}\}_{k=1}^{\infty}$ on the set $K$ \cite{ks96} is defined  by 
 \begin{align*}
h_{top}^{\omega}(F_1^{\infty},K)&=\lim\limits_{\epsilon \to 0}\limsup_{n \to \infty }\frac{1}{n}\log r(n,\omega,\epsilon,K)\\
&=\lim\limits_{\epsilon \to 0}\limsup_{n \to \infty }\frac{1}{n}\log s(n,\omega,\epsilon,K).
 \end{align*}

The definition  of topological entropy  measures   the exponential growth rate of   the number of distinguishable   orbit segments that  can be   detected up to  a given  error, and thus quantitatively reflects  the topological complexity of  an abstract topological dynamical system.

 Now, with the topological entropy defined,  we  turn to the concept of invariance entropy  and draw a comparison between them. Given a  control system $\varSigma=(\mathbb{N},X,U,\mathcal{U},\phi)$, a  set $Q\subset X$ is said to be a \emph{controlled  invariant set} if $Q$ is compact  and  for any $x\in Q$ there exists $\omega_x\in\mathcal{U}$ such that $\phi(\mathbb{N},x,\omega_x)\subset Q$. Fix a controlled invariant set  of $Q$. A  triple $\mathcal C =(\mathcal A, \tau,\nu)$ is  an  \emph{invariant partition} of $Q$ if $\mathcal A=\{A_1,...,A_q\}$ is a finite Borel  partition  of  $Q$, $\tau\in \mathbb{Z}_{+}$, and $\nu: \mathcal A \to U^{\tau}$ is a map such that 
$$\phi(j,A_i,\nu(A_i))\in Q$$
for   all $i=1,2,...,q$ and  each  $j \in [0,\tau] \cap \mathbb{N}$. Let  $\omega_i:=\nu(A_i)$, for $i\in S:=\{1,...,q\}$. Given $n\in \mathbb{N}$ and a  word $s=(s_0,s_1,...,s_{n-1})\in S^n$, we write  $l(s)=n$ to denote the length of $s$. The \emph{concatenation of $n$ controls}  labeled  by the word  $s$ is defined by 
$$\omega_s:=\omega_{s_0}\omega_{s_1}\cdots \omega_{s_{n-1}}.$$ 
The   word $s=(s_0,s_1,...,s_{n-1})$  is called an  \emph{admissible word with length  $n$} if  the set
$$\mathcal{C}_s(Q):=\{x\in Q: \phi(j\tau,x,\omega_{s_0,...,s_{n-1}})\in A_{s_j}, ~\text{for}~{j=0,1,...,n-1}\}$$
is not empty.
The set of all admissible words  with  length  $n$ is  denoted by $\mathcal{L}^n{(\mathcal{C})}$. Notice that if  $\mathcal C =(\mathcal A, \tau,\nu)$ is an invariant partition of $Q$, then for every $x\in Q$ there exists  a  unique   $\mathcal{C}(x) \in S^{\mathbb{N}}$  such that 
$$\phi(j\tau,x,\omega_{\mathcal{C}(x)})\in A_{\mathcal{C}_j(x)},$$
for every $j\in \mathbb{N}$.  For $m\leq n$, let  $\mathcal C_{[m,n]}(x)=(\mathcal{C}_m(x),...,\mathcal{C}_n(x))$. We call   $\mathcal C_{[m,n]}(x)$   the \emph {orbit address}  of $x$ from the time $m$ to $n$ remaining in $Q$. Let  $n\in \mathbb{N}$ and $x\in Q$. The \emph{cylindrical set of $x$  of length $n$}  w.r.t. $\mathcal C$ is defined by 
\begin{align*}
Q_n(x, \mathcal C):&
=\{y\in Q :\mathcal{C}_j(y)=\mathcal{C}_j(x),\text{for}~j=0,1,....,n-1\}\\
&=\{y\in Q :\phi(j\tau,y,\omega_{\mathcal {C}_{[0,n)}(x)})\in A_{\mathcal{C}_j(x) },\text{for}~j=0,1,....,n-1\}\\
&= \cap_{j=0}^{n-1}  \phi_{j\tau, \omega_{\mathcal C_{[0,n)}(x)}}^{-1} (A_{\mathcal{C}_j(x)}),
\end{align*}
where the map  $\phi_{t,\omega}:  X \rightarrow X$, given by  $\phi_{t,\omega}(x):=\phi(t,x,\omega)$, is continuous by the assumption of  (\ref{2.1}).  

 Observe that   $Q_n(x, \mathcal C)$  is the set  of all points in $Q$ that have the same orbit address with $x$ from time $0$ to $n-1$. Therefore,  for  any  two points $x,y\in Q$ with $x\not=y$,  either $Q_n(x, \mathcal C)=Q_n(y, \mathcal C)$ or $Q_n(x, \mathcal C)\cap Q_n(y, \mathcal C)=\emptyset$.  
 
 Each admissible word   with length $n$ naturally defines the  cylindrical sets  for some points  of  $Q$ with the same length $n$,  and each  cylindrical set  with length $n$  is also labeled by an admissible word with length $n$. More precisely,  $ \mathcal{C}_{\omega_{\mathcal C_{[0,n)}(x)}}(Q)=Q_n(x, \mathcal C)$ for each $x\in Q$, and  for each   admissible word $s$ with length $n$,  one has $Q_n(x, \mathcal C)=\mathcal{C}_s(Q)$  since $\C_{[0,n)}(x)=s$ for all $x\in \mathcal{C}_s(Q)$.

\begin{rem}
\begin{enumerate}
\item The cylindrical  set  $Q_n(x, \mathcal C)$  is   in fact an analogue of the   Bowen open  ball  in  the classical case,  which is adapted to the control setting by using invariant partition.
\item  The cylindrical set $Q_n(x, \mathcal C)$  is not  necessarily open since each element of the partition is only Borel  measurable. However,  if  $\mathcal{A}$ is a clopen partition of $Q$, then  the set $Q_n(x, \mathcal C)$ is   both  open and  closed  in $Q$ for all $x\in Q$.  See \cite[Section 8]{whs19}  for the examples of   clopen  invariant partitions.
 \end{enumerate}
\end{rem}

Let $\mathcal C =(\mathcal A, \tau,\nu)$   be an invariant partition of $Q$, $K\subset Q$ and $n\in \mathbb{N}$. A set $E \subset Q$ is a $(\mathcal{C},n,K)$-\emph{spanning set} if for any $x \in K$, there exists $y\in E$ such that $\C_{[0,n)}(x)=\C_{[0,n)}(y)$. The smallest cardinality of a $(\mathcal{C},n,K)$-spanning set is denoted by $r(\mathcal{C},n,K)$. A  set $F \subset K$ is  a $(\mathcal{C},n,K)$-\emph{separated  set} if for any $x, y \in F$ with $x \neq y$, $\C_{[0,n)}(x)\not = \C_{[0,n)}(y)$ (that is, $Q_n(x,\mathcal{C})\cap Q_n(y,\mathcal{C})=\emptyset $).  The largest cardinality of  a $(\mathcal{C},n,K)$-separated set is denoted by $s(\mathcal{C},n,K)$. The  \emph{invariance entropy} of $K$ w.r.t. $\mathcal{C}$ \cite{ck,whs19}  is defined  by 
$$h_{inv}(K,Q,\mathcal{C})=\limsup_{n \to \infty }\frac{1}{n}\log r(\mathcal{C},n,K)=\limsup_{n \to \infty }\frac{1}{n}\log s(\mathcal{C},n,K).$$

Although  invariance entropy is formulated in a similar fashion,  the  control system   pays  more attention to  how many   control functions are used for  achieving  the  certain  control task,  of   keeping  a subset of  a  controlled invariant set invariant,  rather than the distinguishable   orbit segment itself.  For instance,   control systems with (outer) zero invariance entropy are in general   easier to realize the control  tasks since the  needed number  of   control functions is always  uniformly  bounded for all time $n$. Some interesting  work  toward this direction  can be found in \cite{wyc19, zhc21,zhc23}.  To sum up, the  two types of  entropies undertake   different  roles   in capturing the complexity of systems, which   have  the  special meanings for  understanding  the dynamical behaviors of the systems. 
Although entropy theory of the classical topological dynamical systems is  well-developed, how to inject the   ideas  and develop the techniques  from the classical   ones    to understand  the   dynamics of   control systems is still a prolonged  work.

\section{Bowen's equation for  upper capacity invariance pressure}
In this section, we introduce the induced invariance pressure on partitions in subsection 3.1, and prove  Theorem \ref{thm 1.1.1} in subsection 3.2. The  proof of Theorem \ref{thm 1.1} is  given in subsection 3.3.

\subsection{Induced invariance pressure on partitions}
 In this subsection, we first review the  concept of induced topological  pressure in classical dynamical systems,  and  then introduce  the corresponding notion  on   partitions to extend several existing  invariance  pressure-like quantities in discrete-time control systems.

We present the  notions of topological pressure and induced topological  pressure  in topological dynamical systems \cite{w82,jms14,xc}.   Let $(X,f)$  be  a  topological dynamical system  with a   continuous self-map  $f$ on  the  compact metric space $X$.  Denote by $C(X,\mathbb{R})$ the space  of   real-valued continuous maps on $X$  equipped with  the supremum norm. For $n\in\mathbb{N}$ and $\varphi \in  C(X,\mathbb{R})$, we write
$S_n\varphi(x):=\sum_{j=0}^{n-1}\varphi(f^jx)$ to denote the   sum of the energy of $\varphi$ along the orbit $\{x,f(x),...,f^{n-1}(x)\}$.  The \emph{topological  pressure} of $\varphi$  is given by

$$P(f,\varphi)=\lim\limits_{\epsilon \to 0}\limsup_{n \to \infty }\frac{1}{n}\log \sup_{F_n}\{\sum_{x\in F_n}e^{S_n\varphi(x)}\},$$
where  the supremum  is  taken over  the set of  all $(n,\epsilon)$-separated sets of $X$.   Topological pressure of $\varphi$ reflects  how chaotic the system is, which  measures  the exponential growth rate of   the  total energy of   distinguishable   orbit segments along a continuous potential.   The role of $\varphi$ in  some physical systems is quite clear. For instance, it can represent pressure  or temperature to  determine the  possible status of the physical system that is forced   by a  certain rule  over a long period.  To attain  the  purpose of   a more precise measurement of chaos of systems,   a basic strategy is   introducing a scaling function  as an observable function  to split the whole phase space $X$ into a proper finite  partition that depends on the time. Then  the total energy can be computed by summing the energy on  each set of the partition. This can be clarified as follows  in a specific way.

Let  $\psi\in C(X,\mathbb{R})$ with $\psi >0$ be a  scaling function. Given  $T>0$, let $\hat{S}_T$   denote the  set of  time $n$  so that  there exists $x\in X$ such that  $S_{n}\psi(x)\leq T $ and $ S_{(n+1)}\psi(x)>T$. For $n\in  \hat{S}_T$, we define 
$$\hat{X}_n=\{x\in X: S_{n}\psi(x)\leq T~\text{and}~  S_{(n+1)}\psi(x)>T \}.$$
It is easy to see that  the  collection $\{\hat{X}_n\}_{n \in  \hat{S}_T}$ of Borel sets of $X$ is pairwise disjoint  and  $X=\cup_{n\in  \hat{S}_T} \hat{X}_n$. Put
\begin{align*}
	P_{\psi,T}(f,\varphi,\epsilon)
	= \sup\{\sum_{n\in \hat{S}_T}\sum_{x \in F_n}\limits e^{S_n\varphi(x)}\},
\end{align*}
where $F_n$ is an  $(n,\epsilon)$-separated set  of $X_n$.
The \emph{$\psi$-induced topological pressure of $\varphi$} is defined by
$$P_{\psi}(\varphi)=\lim\limits_{\epsilon \to 0}\limsup_{T \to \infty}\frac{1}{T}\log 	P_{\psi,T}(f,\varphi,\epsilon).$$

To sum up, the function $\psi$ is responsible for partitioning the  whole phase space,  and the function $\varphi$ is  responsible for collecting the energy on  set $ \hat{X}_n$.  In the setting of  infinite
conformal iterated function systems \cite{jk11},  $\psi$ is   the geometric potential associated to the cIFS, and  $\varphi$ is the potential defining the level sets under consideration; the  free energy function  given in \cite{jk11} coincides  with  the (special) induced topological pressure.  See also \cite{jms14} for  more  interesting examples of  induced topological pressure of subshift of  finite type.

In fact,  the concept of induced topological pressure is    a generalized notion of pressure-like quantities in  dynamical systems.  

\begin{rem}
\begin{itemize}
\item [(1)] When $\psi=1$,  it is  reduced to  the topological pressure $P(f,\varphi)$;
\item [(2)] When $\psi=1$ and $\varphi=0$,  it  is reduced to the topological entropy $h_{top}(f)$; 
\item [(3)] When $\varphi=0$,  it  is reduced to  the BS dimension ${\rm dim}_{BS,\psi}(X)$ of $X$(cf. \cite[Proposition 4.1]{xc}). In particular, if  $f$ is a $C^{1+\epsilon}$ conformal
expanding map on $X$ and $\psi(x)=\log||d_xf||$,  it is  reduced to the Hausdorff dimension  of $X$.
\end{itemize}
 
\end{rem}

We  are now in a position to inject the ideas  of  \cite{w82,jms14,xc}  into  discrete-time control systems to define  induced invariance pressure.

Throughout the rest of  this paper, for  $\omega \in U^n$, $n\in \N$ and  $\varphi,\psi \in C(U,\mathbb{R})$ with $\psi >0$, we define  $S_n\varphi(\omega):=\sum_{j=0}^{n-1}\varphi(\omega_i)$, $m:=\min_{u\in U}\psi(u)$ and $||\psi||:=\max_{u\in U}|\psi(u)|$.    
Let $\mathcal C =(\mathcal A, \tau,\nu)$  be  an invariant partition of $Q$. Put 
$$m(\varphi,Q,n,\C)=\sup_{F_n}\{\sum_{x\in F_n}e^{S_{n\tau}\varphi({\omega_{\C_{[0,n)}(x)}})}\},$$
where  the supremum  ranges over all $ (\C,n,Q)$-separated sets. 
\begin{df}\label{def 2.1}
The upper capacity invariance  pressure  of $\varphi$ on $Q$ w.r.t. $\mathcal{C}$  is defined by 
$$P_{inv}(\varphi,Q,\mathcal{C})=\limsup_{n \to \infty }\frac{1}{n\tau}\log m(\varphi,Q,n,\C).$$
\end{df} 

Given $T>0$ and $\psi \in C(U,\mathbb{R})$ with $\psi>0$, let $S_T$   denote the  set of  time $n$ such  that  there exists $x\in Q$ such that  $S_{n\tau}\psi({\omega_{\C_{[0,n)}(x)}})\leq T\tau$ and $ S_{(n+1)\tau}\psi({\omega_{\C_{[0,n+1)}(x)}})>T\tau$.
For  each $n\in S_T$,  we define the set
$$X_n=\{x\in Q:  S_{n\tau}\psi({\omega_{\C_{[0,n)}(x)}})\leq T\tau~ \mbox{and}~ S_{(n+1)\tau}\psi({\omega_{\C_{[0,n+1)}(x)}})>T\tau\}.$$
Clearly, the set $X_n$ is a  (finite) union of some cylindrical  sets $Q_{n+1}(x,\C)$ for some $x$. Put 
\begin{align*}
&P_{inv,\psi, T}(\varphi,Q,\C)\\
= &\sup\left\{\sum_{n\in S_T}\sum_{x \in F_n}\limits e^{\S}: F_n \mbox{ is a}~ (\C,n,X_n)\mbox{-separated set} \right\}.
\end{align*}

\begin{df}
The $\psi$-induced  invariance pressure  of $\varphi$ on $Q$ w.r.t. $\mathcal{C}$ is defined by 
$$P_{inv,\psi }(\varphi,Q,\C)=\limsup_{T \to \infty}\frac{1}{T\tau}\log P_{inv,\psi, T}(\varphi,Q,\C).$$
\end{df}

\begin{rem}\label{rem 3.4}
\begin {enumerate}
\item  If $S_T\not=\emptyset$, then for each $n\in S_T$,  we have $\frac{T}{||\psi||} -1<n\leq \frac{T}{m}$, which shows  $S_T$ is a finite set.
\item If $\psi =1$,  noting that $S_T$ is a singleton, then  $P_{inv,1 }(\varphi,Q,\C)= P_{inv}(\varphi,Q,\mathcal{C}).$
\item For every  $\psi \in C(U,\R)$  with $\psi >0$, the function $$P_{inv,\psi }(\cdot,Q,\C): C(U.\R)\rightarrow \R$$ is finite.
Indeed,   for sufficiently large  $T>0$, we have  $P_{inv,\psi, T}(\varphi,Q,\C)$
$\geq e^{-\frac{T}{m}\tau ||\varphi||}$ and hence $P_{inv,\psi }(\varphi,Q,\C)\geq -\frac{1}{m}||\varphi||>-\infty.$ On the other hand, noticing that for each $n\in S_T$  the set $X_n$ is a union of some cylindrical sets $Q_{n+1}(x,\C)$, this means  if $F_n$ is a $(\C,n,X_n)$-separated set, then $\#F_n\leq ( \#\mathcal{A})^{n+1}$.  Therefore,
\begin{align*}
\sum_{n\in S_T}\sum_{x \in F_n}\limits e^{\S}&\leq\sum_{n\in S_T} ( \#\mathcal{A})^{n+1} e^{n\tau||\varphi||}\\
&\leq (\frac{T}{m}-\frac{T}{||\psi||}+1) (\#\mathcal{A})^{\frac{T}{m}+1} e^{\frac{T}{m}\tau||\varphi||}.
\end{align*}
This gives us  $P_{inv,\psi, T}(\varphi,Q,\C)\leq (\frac{T}{m}-\frac{T}{||\psi||}+1) (\#\mathcal{A})^{\frac{T}{m}+1} e^{\frac{T}{m}\tau||\varphi||}$. Consequently, 
$$P_{inv,\psi }(\varphi,Q,\C)\leq \frac{\log\#\mathcal{A}}{m\tau}+\frac{1}{m}||\varphi||<\infty.$$
\end{enumerate}
\end{rem}

The following proposition formulates  an equivalent definition for  induced invariance pressure using spanning sets.

Let 
\begin{align*}
	&Q_{inv,\psi, T}(\varphi,Q,\C)\\
	= &\inf\left\{\sum_{n\in S_T}\sum_{x \in E_n}\limits e^{\S}: E_n \mbox{ is a}~ (\C,n,X_n)\mbox{-spanning set}\right\}.
\end{align*}

\begin{prop}\label{poro 3.4}
Let $Q$ be a  controlled  invariant set and  $\mathcal C =(\mathcal A, \tau,\nu)$ be an invariant partition of  $Q$, and let $\varphi,\psi \in C(U,\mathbb{R})$ with $\psi>0$. 
Then $$P_{inv,\psi }(\varphi,Q,\C)=\limsup_{T \to \infty}\frac{1}{T\tau}\log Q_{inv,\psi, T}(\varphi,Q,\C).$$ 
\end{prop}

\begin{proof}
Let $T>0$.  Notice that for each $n\in S_T$, a  $(\C,n,X_n)$-separated  set  with the  maximal cardinality is also a $(\C,n,X_n)$-spanning set. This shows that
$$P_{inv,\psi }(\varphi,Q,\C)\geq \limsup_{T \to \infty}\frac{1}{T\tau}\log Q_{inv,\psi, T}(\varphi,Q,\C).$$

Now,  let $E_n$ be a $(\C,n,X_n)$-spanning set and  $F_n$ be a $(\C,n,X_n)$-separated set.  Consider the map $\Phi: F_n \rightarrow E_n$ by assigning each $x\in F_n$ to $\Phi(x)\in E_n$ such that  $\C_{[0,n)}(\Phi(x))=\C_{[0,n)}(x)$. Then $\Phi$ is injective.
Therefore,
\begin{align*}
&\sum_{n\in S_T}\sum_{y \in E_n}\limits e^ {S_{n\tau}\varphi({\omega_{\C_{[0,n)}(y)}})}\\
\geq&\sum_{n\in S_T}\sum_{x \in F_n}\limits e^ {S_{n\tau}\varphi({\omega_{\C_{[0,n)}(\Phi(x))}})}\\
=&\sum_{n\in S_T}\sum_{x \in F_n}\limits e^ {S_{n\tau}\varphi({\omega_{\C_{[0,n)}(x)}})}, \text{by}~ \C_{[0,n)}(\Phi(x))=\C_{[0,n)}(x).
\end{align*}
 It follows that 
$$P_{inv,\psi }(\varphi,Q,\C)\leq\limsup_{T \to \infty}\frac{1}{T\tau}\log Q_{inv,\psi, T}(\varphi,Q,\C).$$

\end{proof}

\begin{rem}
By Remark   \ref{rem 3.4}, (2) and  Proposition  \ref{poro 3.4},  one has an   equivalent  definition  for  $P_{inv}(\varphi,Q,\mathcal{C})$   given by spanning sets, that is,
$$P_{inv}(\varphi,Q,\mathcal{C})=\limsup_{n \to \infty }\frac{1}{n\tau}\log \inf_{E_n}\{\sum_{x\in E_n}e^{S_{n\tau}\varphi({\omega_{\C_{[0,n)}(x)}})}\}$$
with  the infimum  is taken  over  all $ (\C,n,Q)$-spanning sets. 

Recall that    the  upper capacity invariance pressure \cite[Section 4]{zh19} on invariant partition  using admissible words  is defined by 
$$P_{inv}^{*}(\varphi,Q,\mathcal{C}):=\limsup_{n \to \infty }\frac{1}{n\tau}\log \inf_{\mathcal{G}}\{\sum_{s\in \mathcal{G}} e^{S_{n\tau}\varphi(\omega_s)}\},$$
where the infimum is taken over  all finite  families $\mathcal{G}\subset \mathcal{L}^n(\mathcal C)$  such that $\cup_{s\in \mathcal{G}}\mathcal{C}_s(Q) \supset Q$.	
 Using the facts $ \mathcal{C}_{\omega_{\mathcal C_{[0,n)}(x)}}(Q)=Q_n(x, \mathcal C)$ for each $x\in Q$, and  for each   admissible word $s$ with length $n$,  one has $Q_n(x, \mathcal C)=\mathcal{C}_s(Q)$  since $\C_{[0,n)}(x)=s$ for all $x\in \mathcal{C}_s(Q)$, it is  readily shown that  $P_{inv}^{*}(\varphi,Q,\mathcal{C})=P_{inv}(\varphi,Q,\mathcal{C})$.
\end{rem}

\subsection{A dimensional characterization  for $P_{inv,\psi }(\varphi,Q,\C)$}
 To establish Bowen's equation for upper capacity invariance pressure on partitions,  we   formulate   an equivalent characterization for  $P_{inv}(\varphi,Q,\mathcal{C})$  whose definition  has  a dimensional flavor,  which  plays a critical role in  our later proofs.
 
For  each $T>0$,  we define the (one-sided) time set
$$G_T:=\{n\in \mathbb{N}: \exists x \in Q ~\mbox{such that } S_{n\tau}\psi({\omega_{\C_{[0,n)}(x)}})>T\tau\}.$$
Given  $n\in G_T$, we put
$$Y_n=\{x\in Q: S_{n\tau}\psi({\omega_{\C_{[0,n)}(x)}})>T\tau\}.$$
and
\begin{align*}
&R_{inv,\psi, T}(\varphi,Q,\C)\\
= &\sup\left\{\sum_{n\in G_T}\sum_{x \in F_n^{'}}\limits e^ {\S}:F_n ^{'}\mbox{ is a}~ (\C,n,Y_n)\mbox{-separated set} \right\}.
\end{align*}

 The  dimensional characterization of  induced topological pressure is  presented in  Theorem \ref{thm 1.1.1}.
\begin{proof}[Proof of Theorem \ref{thm  1.1.1}]
For any   $n\in \mathbb{N}$ and any $ x\in Q$,  we define $m_n(x)$ as the unique positive integer  satisfying that 
\begin{align}\label{equ  2.3}
(m_n(x)-1)||\psi||\tau<S_{n\tau}\psi({\omega_{\C_{[0,n)}(x)}})\leq m_n(x)||\psi||\tau.
\end{align}
It is easy to check that for any $x\in Q$,  the inequality 
\begin{align}\label{equ 2.4}
e^{-\beta ||\psi||\tau m_n(x)}e^{-|\beta|||\psi||\tau}\leq e^{-\beta \S}\leq e^{-\beta ||\psi||\tau m_n(x)}e^{|\beta|||\psi||\tau}
\end{align}
 holds for any $\beta \in\mathbb{R}$. Define 
\begin{align*}
R_{inv,\psi, T}^{*}(\beta):=\sup\{\sum_{n\in G_T}\sum_{x \in F_n^{'}}\limits e^ {\S-\beta||\psi||\tau m_n(x)} \},
\end{align*}
where $F_n ^{'}$  is a  $(\C,n,Y_n)$-separated set. 
By   inequality (\ref{equ 2.4}), to show the equality 
\begin{align*}
P_{inv,\psi }(\varphi,Q,\C)=\inf\{\beta \in \mathbb{R}: \limsup_{T \to \infty}R_{inv,\psi, T}(\varphi-\beta\psi,Q,\C)<\infty\},
\end{align*}
 it suffices to verify 
\begin{align*}
P_{inv,\psi }(\varphi,Q,\C)
=\inf\{\beta \in \mathbb{R}: \limsup_{T \to \infty}R_{inv,\psi, T}^{*}(\beta)<\infty\}.
\end{align*}
For simplifying the  notions, we let  $$\text{lhs}:=P_{inv,\psi }(\varphi,Q,\C)$$ and
$\text{rhs}:=\inf\{\beta \in \mathbb{R}: \limsup_{T \to \infty}R_{inv,\psi, T}^{*}(\beta)<\infty\}.$

We firstly show {\rm lhs  $\leq$ rhs}.  Let  $\beta<P_{inv,\psi }(\varphi,Q,\C)$.  Choose $\delta>0$ and  a  subsequence $\{T_j\}_{j\in \mathbb{N}}$ that   converges  to  $\infty$ as $j \to \infty$  satisfying
 $$\beta+\delta<P_{inv,\psi }(\varphi,Q,\C),$$
and  $$P_{inv,\psi }(\varphi,Q,\C)=\lim_{j \to \infty}\frac{1}{T_j\tau}\log P_{inv,\psi, T_j}(\varphi,Q,\C).$$
Then there exists  $J_0$ such that for any $j\geq J_0$, one can choose    a $(\C,n,X_n)$-separated set $F_n$ with  $n\in S_{T_j}$ satisfying 
\begin{align}\label{equ 2.5}
e^{T_j\tau(\beta+\delta)}<\sum_{n\in S_{T_j}}\sum_{x\in F_n}e^{\S}.
\end{align}  
Noting  that $T_j \to \infty$,  we   can   choose   a subsequence $\{T_{j_{k}}\}_{k \geq 1}$ such that 
$$\frac{T_{j_k}}{m}+1<\frac{T_{j_{k+1}}}{||\psi||}-1.$$
Without loss of generality, we still denote  the  subsequence $\{T_{j_k}\}_{k\geq 1}$ by $\{T_j\}_{j\geq 1}$.    Then the time sets $S_{T_i}\cap S_{T_j}=\emptyset$ for any $i,j \geq  J_0$ with $i\not= j$.
For each $j\geq J_0$ and $n\in S_{T_j}$, we have $$(T_j-||\psi||)\tau<S_{n\tau}\psi({\omega_{\C_{[0,n)}(x)}})\leq T_j\tau$$ for all $x\in F_n$.  Together with  the inequality (\ref{equ 2.3}), we get 
\begin{align}\label{equ 2.6}
|||\psi||\tau m_n(x)-T_j\tau|<2||\psi||\tau.
\end{align}
and hence $-\beta ||\psi||\tau m_n(x)\geq -\beta T_j\tau -2|\beta|\tau||\psi||.$ 

Therefore,
 \begin{align*}
R_{inv,\psi, T}^{*}(\beta)
\geq& \sum_{j\geq J_0,\atop T_j-||\psi||>T}\sum_{n\in S_{T_j}}\sum_{x\in F_n}
e^{\S-\beta||\psi||\tau m_n(x)}\\  
\geq& e^{-2|\beta|\tau||\psi||}\sum_{j\geq J_0,\atop T_j-||\psi||>T}\sum_{n\in S_{T_j}}\sum_{x\in F_n}
e^{\S-\beta T_j\tau}\\
\geq& e^{-2|\beta|\tau||\psi||}\sum_{j\geq J_0, T_j-||\psi||>T}
e^{\delta T_j\tau },~~~~~~~ \text{by~~(\ref{equ 2.5}) }\\
=&\infty.
\end{align*}
This leads to 
\begin{align}\label{equ 2.7}
\limsup_{T \to \infty}R_{inv,\psi, T}^{*}(\beta)=\infty,
\end{align}
which implies that {\rm lhs $\leq$ rhs}.

We proceed to show {\rm lhs $\geq$ rhs}.  Let $\delta >0$ and  then choose an $l_0\in \mathbb{N}$ such that   for  any $l\geq l_0$, 
\begin{align}\label{equ 2.8}
&P_{inv,\psi, lm}(\varphi,Q,\C)<e^{ lm\tau(P_{inv,\psi }(\varphi,Q,\C)+\frac{\delta}{2})},
\end{align}
where   $m=\min_{u\in U}\limits \psi(u)>0.$  Let $\beta :=P_{inv,\psi }(\varphi,Q,\C)+\delta$. For each $l\geq l_0$ and   $n\in S_{lm}$, similar to  (\ref{equ 2.6}) if $F_n$  is a  $(\C,n,X_n)$-separated set, then 
$$|||\psi||\tau m_n(x)-lm\tau|<2||\psi||\tau$$
and
\begin{align} \label{equ 2.9}
-\beta||\psi||\tau m_n(x) 
\leq -\beta lm\tau +2||\psi||\tau\cdot |\beta|. 
\end{align}
For  sufficiently large  $T>l_0m$,  if $n\in G_T$ and  $F_n^{'}$  is a $(\C,n,Y_n)$-separated set, then for each $x\in F_n^{'}$, there exists a  unique $l> l_0$ such that $(l-1)m\tau<S_{n\tau}\psi({\omega_{\C_{[0,n)}(x)}})\leq lm\tau$. Then  we  get $$S_{(n+1)\tau}\psi({\omega_{\C_{[0,n+1)}(x_1)}})=S_{n\tau}\psi({\omega_{\C_{[0,n)}(x)}})+S_{\tau}\psi({\omega_{\C_n(x)}})>lm\tau.$$  This shows that $n\in S_{lm}$ and $x\in X_n$.  
Therefore, 
\begin{align*}
R_{inv,\psi, T}^{*}(\beta)
\leq &\sum_{l> l_0}\sup\sum_{n\in S_{lm}}\sum_{x\in F_n}
e^{\S-\beta||\psi||\tau m_n(x)}\\
\leq& e^{2||\psi||\tau\cdot |\beta|}
\sum_{l> l_0}\sup\sum_{n\in S_{lm}}\sum_{x\in F_n}
e^{{\S-\beta lm\tau }}~~ \text{by (\ref{equ 2.9})} \\
\leq& e^{2||\psi||\tau\cdot |\beta|}
\sum_{l> l_0}e^{-\frac{\delta}{2}m\tau l} ~~ \text{by (\ref{equ 2.8})}.
\end{align*}
This yields that  $\limsup_{T \to \infty}R_{inv,\psi, T}^{*}(\beta)<\infty$.
Letting $\delta \to 0$, we get  {\rm lhs $\geq$ rhs}.
\end{proof}  
\subsection{Proof of Theorem 1.2}
Now, we give the proof of  Theorem \ref{thm 1.1}.
\begin{proof}[Proof of Theorem \ref{thm 1.1}]
We divide the proof into two steps.

{Step 1.} we show $P_{inv}(\varphi-\beta \psi,Q,\mathcal{C})=0$ has a  unique root. 

Consider the  map $$\Phi:\beta \in \mathbb{R} \longmapsto \Phi(\beta):=P_{inv}(\varphi-\beta \psi,Q,\mathcal{C}).$$
By Remark \ref{rem 3.4},  $\Phi(\beta)$ is finite for all  $\beta \in \R$.

We show $\Phi$ is  a continuous and strictly decreasing function on $\mathbb{R}$. Let $\beta_1,\beta_2\in \mathbb{R}$, $n\in \N$ and $F_n$ be a $(\C,n,Q)$-separated set. Then 
\begin{align*}\label{inequ 2.10}
&\sum_{x\in F_n} e^{{S_{n\tau}(\varphi -\beta_2 \psi)(\omega_{\C_{[0,n)}(x)})}-|\beta_1-\beta_2|n\tau\cdot||\psi||}\\
\leq&\sum_{x\in F_n}  e^{{S_{n\tau}(\varphi -\beta_1 \psi)(\omega_{\C_{[0,n)}(x)})}}\\
\leq&\sum_{x\in F_n} e^{{S_{n\tau}(\varphi -\beta_2 \psi)(\omega_{\C_{[0,n)}(x)})}+|\beta_1-\beta_2|n\tau\cdot||\psi||},
\end{align*}
which implies that 
\begin{align}
\Phi(\beta_2)-|\beta_1-\beta_2|||\psi||\leq \Phi(\beta_1)
\leq \Phi(\beta_2)+|\beta_1-\beta_2|||\psi||.
\end{align}
Then  the continuity of   $\Phi$ follows from   the    inequality 
$$|\Phi(\beta_1)-\Phi(\beta_2)|\leq||\psi||\cdot |\beta_1-\beta_2|.$$
Let  $\beta_1, \beta_2 \in \mathbb{R}$  with $\beta_1<\beta_2$.  One can similarly obtain  that 
\begin{align}\label{equ 2.11}
\Phi(\beta_2)\leq \Phi(\beta_1)-(\beta_2-\beta_1)m.
\end{align}
So  the map $\Phi$ is strictly  decreasing. 

According to  the possible values of $P_{inv}(\varphi,Q,\mathcal{C})$, we have the following three cases.
\begin{enumerate}
\item If   $P_{inv}(\varphi,Q,\mathcal{C})=0$, then $0$ is  exactly the unique root of the equation  $\Phi(\beta)=0$. 
\item  If   $P_{inv}(\varphi,Q,\mathcal{C})>0$,  taking $\beta_1=0$  and $\beta_2=h>0$ in (\ref{equ 2.11}),
then 
$$P_{inv}(\varphi-h \psi,Q,\mathcal{C})\leq P_{inv}(\varphi,Q,\mathcal{C})-hm.$$
The intermediate value theorem for  continuous function and the strict-decreasing   property    ensure that  the equation  $\Phi(\beta)=0$  has  the  unique root  $\beta^{*}$  satisfying $0<\beta^{*} \leq \frac{1}{m}P_{inv}(\varphi,Q,\mathcal{C})$. 
\item  If   $P_{inv}(\varphi,Q,\mathcal{C})<0$,  taking $\beta_1=h<0$  and $\beta_2=0$ in (\ref{equ 2.11}) again,
then 
$$P_{inv}(\varphi,Q,\mathcal{C})-hm\leq P_{inv}(\varphi-h\psi,Q,\mathcal{C}).$$
Similarly,  we  know  that  the equation $\Phi(\beta)=0$ has the  unique root  $\beta^{*}$  satisfying  
$ \frac{1}{m}P_{inv}(\varphi,Q,\mathcal{C})\leq\beta^{*}<0.$
\end{enumerate}
To sum up,  the equation  $\Phi(\beta)=0$ has a unique (finite) root.  

{Step 2.}  we show $P_{inv,\psi }(\varphi,Q,\C)$ is the unique  root of  the equation $\Phi(\beta)=0$.

Since $\Phi$ is   continuous and strictly decreasing,  we have
\begin{align}\label{equ 3.11}
\begin{split}
&\inf\{\beta\in \mathbb{R}:P_{inv}(\varphi-\beta \psi,Q,\mathcal{C})<0\}\\
=&\inf\{\beta\in \mathbb{R}:P_{inv}(\varphi-\beta \psi,Q,\mathcal{C})\leq0\}\\
=&\sup\{\beta\in \mathbb{R}:P_{inv}(\varphi-\beta \psi,Q,\mathcal{C})\geq0\}.
\end{split}
\end{align}
Hence it suffices to show
\begin{align}\label{equ 3.12}
P_{inv,\psi }(\varphi,Q,\C)=\inf\{\beta\in \mathbb{R}:P_{inv}(\varphi-\beta \psi,Q,\mathcal{C})\leq0\}.
\end{align}

We first show  
\begin{align}\label{inequ 2.12}
P_{inv,\psi }(\varphi,Q,\C)\geq\inf\{\beta\in \mathbb{R}:P_{inv}(\varphi-\beta \psi,Q,\mathcal{C})\leq0\}.
\end{align}
By Theorem \ref{thm 1.1.1}, we need to  show  for any 
$\beta \in  \{s \in \mathbb{R}: \limsup_{T \to \infty}\limits R_{inv,\psi, T}(\varphi-s\psi,Q,\C)<\infty\},$  one has
$$P_{inv}(\varphi-\beta \psi,Q,\mathcal{C})\leq0.$$
Let $M:=\limsup_{T \to \infty} \limits R_{inv,\psi, T}(\varphi-\beta\psi,Q,\C)$. Then  there exists $T_0\in \mathbb{N}$  so that for all $T\geq T_0$,  
$$R_{inv,\psi, T}(\varphi-\beta\psi,Q,\C)<M+1.$$
Using  the Definition \ref{def 2.1}, one can  choose a subsequence $\{n_j\}_{j \geq 1}$ that   converges  to  $\infty$ as $j \to \infty$  such that
\begin{align*}
P_{inv}(\varphi-\beta\psi,Q,\mathcal{C})=\lim_{j\to \infty }\frac{1}{n_j\tau}\log m(\varphi-\beta\psi,Q,n_j,\C).
\end{align*}
Fix  $T\geq T_0$. Then there exists  sufficiently large $n_j>T$ such that 
$S_{n_j\tau}\psi({\omega_{\C_{[0,n_j)}(x)}})>T\tau$ for all $x\in Q$. So $n_j \in G_T$ and $Y_{n_j}=Q$.  Let $F_{n_j}$  be  a  $(\C,n_j,Q)$-separated set. Then 
$$\sum_{x\in F_{n_j}}\limits e^{S_{n_j\tau}(\varphi -\beta \psi)(\omega_{\C_{[0,n_j)}(x)})}<M+1,$$
which yields that $m(\varphi-\beta\psi,Q,n_j,\C)\leq M+1$. This shows $P_{inv}(\varphi-\beta \psi,Q,\mathcal{C})\leq0.$

By   (\ref{equ 3.11}), we  continue to show
\begin{align}\label{inequ 2.13}
P_{inv,\psi }(\varphi,Q,\C)\leq\inf\{\beta\in \mathbb{R}:P_{inv}(\varphi-\beta \psi,Q,\mathcal{C})<0\}.
\end{align} 
Let $\beta \in \mathbb{R}$ with $P_{inv}(\varphi-\beta \psi,Q,\mathcal{C})=2a<0.$
Then there is $N_0$ such that for all $n\geq N_0$, 
$$\sup\left\{\sum_{x\in F_n}\limits e^{S_{n\tau}(\varphi-\beta \psi)({\omega_{\C_{[0,n)}(x)}})}: F_n ~\mbox{is a } (\C,n,Q)\mbox{-separated set}\right\}<e^{an\tau}.$$
Fix  sufficiently  large $T$ so that for each  $n\in G_T$  one has  $n\geq N_0$. Then   
\begin{align*}
R_{inv,\psi, T}(\varphi-\beta\psi,Q,\C)&\leq\sum_{n\geq N_0}\sup_{F_n}\sum_{x\in F_n}\limits e^{S_{n\tau}(\varphi-\beta \psi)({\omega_{\C_{[0,n)}(x)}})}\\
&<\sum_{n\geq N_0}e^{an\tau}<\infty,
\end{align*}
where $F_n$ is a  $ (\C,n,Q)$-separated set.
Hence $\limsup_{T \to \infty}R_{inv,\psi, T}(\varphi-\beta\psi,Q,\C)<\infty$ and 
\begin{align*}
\begin{split}
&\inf\{\beta\in \mathbb{R}:P_{inv}(\varphi-\beta \psi,Q,\mathcal{C})<0\}\\
\geq &\inf\{\beta \in \mathbb{R}: \limsup_{T \to \infty}R_{inv,\psi, T}(\varphi-\beta\psi,Q,\C)<\infty\}.\\
=&P_{inv,\psi }(\varphi,Q,\C) ,~~~~\text{by Theorem \ref{thm 1.1.1}}.
\end{split}
\end{align*}
By  inequalities (\ref{inequ 2.12}) and (\ref{inequ 2.13}),  we get (\ref{equ 3.12}).
\end{proof}

\section{Bowen's equation for  Pesin-Pitskel invariance pressure}
In this section, we  give the proofs of Theorem \ref{thm 1.2} in subsection 4.1 and Theorem \ref{thm 1.3} in subsection 4.2.
\subsection{Pesin-Pitskel invariance pressure on subsets} In this  subsection, we first recall the definition of Pesin-Pitskel invariance pressure on subsets \cite{zh19,z20},  and then introduce   a new  notion called  BS invariance dimension to establish Bowen's equation for Pesin-Pitskel  invariance pressure.

Let $Q$ be a  controlled  invariant set and  $\mathcal C =(\mathcal A, \tau,\nu)$ be an invariant partition of  $Q$. Let $\varphi\in C(U,\mathbb{R})$, $Z\subset Q, \lambda\in \R, N\in \N$.  Put
\begin{align} \label{equ 3.1}
M_{\C}(\varphi,Z,Q,\lambda,N)=\inf\left\{\sum_{i\in I}\limits  e^{-\lambda n_i \tau+\M}\right\},
\end{align}
where the infimum  is taken over all  finite or countable families $\{Q_{n_i}(x_i,\C)\}_{i\in I}$ such that $x_i \in Q$,  $n_i \geq N$ and  $\cup_{i\in I}Q_{n_i}(x_i,\C)\supset Z$.

Notice that the quantity $M_{\C}(\varphi,Z,Q,\lambda,N)$ is non-decreasing as $N$ increases. Then the limit
$$M_{\C}(\varphi,Z,Q,\lambda)=\lim_{N \to \infty }M_{\C}(\varphi,Z,Q,\lambda,N)$$
exists. It is easy  to check that  $M_{\C}(\varphi,Z,Q,\lambda)$  jumps from $\infty$ to $0$ at a critical value of the parameter $\lambda$.  
\begin{df}
The   Pesin-Pitskel invariance pressure of  $\varphi$ on $Z$ w.r.t. $\C$  is defined by  the critical value: 
\begin{align*}
P_{\C}(\varphi,Z,Q):&=\inf\{\lambda:M_{\C}(\varphi,Z,Q,\lambda)=0\}\\
&=\sup\{\lambda:M_{\C}(\varphi,Z,Q,\lambda)=\infty\}.
\end{align*}
The Bowen invariance entropy  of  $Z$ w.r.t.  $\C$  \cite{hz18,whs19} is given  by ${\rm dim}_{\C}(Z,Q):=P_{\C}(0,Z,Q)$.
\end{df}
\begin{rem}
Actually, Pesin-Pitskel invariance pressure  can be alternatively defined  by using  admissible words \cite{zh19,z20}. Put 
$$\Lambda_{\C}(\varphi,Z,Q,\lambda,N)=\inf_{\mathcal G}\left\{\sum_{i\in I}\limits  e^{-\lambda l(s) \tau+S_{l(s)\tau}\varphi(\omega_s)}\right\},$$
where the infimum  is taken over all  finite or countable  admissible words   such that  $l(s)\geq N$ and  $\cup_{s\in \mathcal{G}}\C_s(Q)\supset Z$.

Let $\Lambda_{\C}(\varphi,Z,Q,\lambda)=\lim_{N \to \infty }\Lambda_{\C}(\varphi,Z,Q,\lambda,N)$. We  define 
\begin{align*}
P_{\C}^{*}(\varphi,Z,Q)&=\inf\{\lambda:\Lambda_{\C}(\varphi,Z,Q,\lambda)=0\}\\
&=\sup\{\lambda:\Lambda_{\C}(\varphi,Z,Q,\lambda)=\infty\}.
\end{align*}

Using the facts  $ \mathcal{C}_{\omega_{\mathcal C_{[0,n)}(x)}}(Q)=Q_n(x, \mathcal C)$ for each $x\in Q$, and  for each   admissible word $s$ with length $n$,  one has $Q_n(x, \mathcal C)=\mathcal{C}_s(Q)$  since $\C_{[0,n)}(x)=s$ for all $x\in \mathcal{C}_s(Q)$ again, it is  readily shown that  $P_{inv}^{*}(\varphi,Q,\mathcal{C})=P_{inv}(\varphi,Q,\mathcal{C})$.
\end{rem}

Fix a non-empty subset $Z\subset Q$,   an invariant partition $\mathcal C =(\mathcal A, \tau,\nu)$  of  $Q$ and $\varphi \in C(U,\R)$.  We  investigate the existence and uniqueness of  the root of  the equation  associated with   Pesin-Pitskel invariance pressure:
$$\Phi(t)=P_{\C}(t\varphi,Z,Q)=0.$$

\begin{prop}\label{prop 3.3}
Let $\varphi\in C(U,\mathbb{R})$ with $\varphi<0$. Then  the equation $\Phi(t)=0$  has  a unique (finite) root $t^{*}$ satisfying
$$-\frac{1}{m}  {\rm dim}_{\C}(Z,Q) \leq t^{*} \leq -\frac{1}{M}{\rm dim}_{\C}(Z,Q),$$
where $m=\min_{u\in U}\varphi(x)$ and $M=\max_{u\in U}\varphi(x)$. 
\end{prop}
\begin{proof}
We show $P_{\C}(\varphi,Z,Q)$ is finite for any $\varphi \in C(U,\R).$   Let  $\varphi \in C(U,\R)$. Then  for any $x\in Q$ and $n\in N$,  one has 
$|\S|\leq n\tau||\varphi||$.
 Since
\begin{align*}
M_{\C}(\varphi,Z,Q,-||\varphi||,N)&=\inf\left\{\sum_{i\in I}\limits  e^{||\varphi||n_i\tau+\cdot\M}\right \}\\
&\geq \inf\left\{\sum_{i\in I}\limits  e^{||\varphi||n_i\tau-n_i\tau||\varphi||}\right\}>1,
\end{align*}
where the infimum  is taken over all  finite or countable families $\{Q_{n_i}(x_i,\C)\}_{i\in I}$ such that $x_i \in Q$,  $n_i \geq N$ and  $\cup_{i\in I}Q_{n_i}(x_i,\C)\supset Z$,  this  implies that $$P_{\C}(\varphi,Z,Q)\geq-||\varphi||>-\infty.$$ 

On  the other hand,    fix a   cover  $\{Q_N(x_i,\C)\}_{i\in I}$ of $Z$ whose cyclinder  sets  with  length $N $ are  at most   $(\#\mathcal{A})^N$. Then   
\begin{align*}
M_{\C}(\varphi,Z,Q,\frac{\log \mathcal{\#A}}{\tau}+||\varphi||,N)\leq\sum_{i\in I}\limits  e^{-(\frac{\log \mathcal{\#A}}{\tau}+||\varphi||)N\tau+N\tau||\varphi||}\leq 1,
\end{align*}
which implies that $P_{\C}(\varphi,Z,Q)\leq  \frac{\log \mathcal{\#A}}{\tau}+||\varphi||<\infty.$

We show $\Phi$ is  a continuous and strictly decreasing. Let $t_1, t_2\in \mathbb{R}$ with $t_1>t_2$.  Given  $N\in \mathbb{N}$  and   a  family  $\{Q_{n_i}(x_i,\C)\}_{i\in I}$  of cyclinder  sets such that $x_i \in Q$,  $n_i \geq N$ and  $\cup_{i\in I}Q_{n_i}(x_i,\C)\supset Z$, one  has
\begin{align*}
&\sum_{i\in I}\limits   e^{-n_i \lambda\tau+t_2\M+(t_1-t_2)n_im\tau}\\
\leq&\sum_{i\in I}\limits  e^{-n_i \lambda\tau+t_1\M}\\
\leq&\sum_{i\in I}\limits   e^{-n_i \lambda\tau+t_2\M+(t_1-t_2)n_iM\tau}.
\end{align*}
This implies that 
\begin{align}\label{4.1}
P_{\C}(t_1\varphi,Z,Q)\leq P_{\C}(t_2\varphi,Z,Q)+(t_1-t_2)M
\end{align}
and
\begin{align}\label{4.2}
P_{\C}(t_2\varphi,Z,Q)+(t_1-t_2)m \leq P_{\C}(t_1\varphi,Z,Q).
\end{align}
So $\Phi$ is  a continuous and strictly decreasing function on $\mathbb{R}$.

Putting  $t_1=h>0$ and $ t_2=0$ in (\ref{4.1}) and noting that  $P_{\C}(0\cdot\varphi,Z,Q)={\rm dim}_{\C}(Z,Q)\geq 0$, we have 
$$P_{\C}(h\varphi,Z,Q)\leq {\rm dim}_{\C}(Z,Q)-h(-M).$$
Again, putting  $t_1=h>0, t_2=0$ in (\ref{4.2}), we have 
$$P_{\C}(h\varphi,Z,Q)\geq {\rm dim}_{\C}(Z,Q)-h(-m).$$
Then the equation $\Phi(t)=0$ has a unique root  $t^{*}$ satisfying  $$-\frac{1}{m}  {\rm dim}_{\C}(Z,Q) \leq t^{*} \leq -\frac{1}{M}{\rm dim}_{\C}(Z,Q).$$
\end{proof}

BS dimension was originally introduced by Barreira and Schmeling  \cite{bs00} to  find the root of the  equation defined by topological pressure of additive  potentials in topological dynamical systems.  In the framework of control systems, we borrow the ideas   used in \cite{bs00,hz18} to define BS invariance dimension on subsets of partitions. The new notion  allows us to obtain the precise root of the equation $\Phi(t)=0$   considered in Proposition \ref{prop 3.3}. 

Let $Q$ be a  controlled  invariant set and  $\mathcal C =(\mathcal A, \tau,\nu)$ be an invariant partition of  $Q$, and let $\varphi\in C(U,\mathbb{R})$ with $\varphi >0$, $Z\subset Q, \lambda\in \R, N\in \N$.  Define 
$$R_{\C}(\varphi,Z,Q,\lambda,N)=\inf\left\{\sum_{i\in I}\limits  e^{-\lambda\M}\right\},$$
where the infimum  is taken over all  finite or countable families $\{Q_{n_i}(x_i,\C)\}_{i\in I}$ such that $x_i \in Q$,  $n_i \geq N$ and  $\cup_{i\in I}Q_{n_i}(x_i,\C)\supset Z$.

It is known  that  Hausdorff   dimension is defined by  geometric radius, that is, the diameter  $|U|$ of  open set $U$.   However, in  topological dynamical systems,   geometric radius  is replaced by the dynamical radius  $e^{-n_i}$ of  Bowen open ball $B_n(x,\epsilon)$  to introduce some  entropy-like quantities. In the case of control systems, we   expect that different ``open balls"  with the same time $n$, but may have different  radius. Therefore, the positive continuous function $\varphi$ plays the certain role by giving  the weight $  e^{-\M}$  for each  different point $x_i$ with the time $n_i$.

Since $R_{\C}(\varphi,Z,Q,\lambda,N)$ is non-decreasing  as $N$ increases, so the limit
$$R_{\C}(\varphi,Z,Q,\lambda)=\lim_{N\to \infty}R_{\C}(\varphi,Z,Q,\lambda,N).$$
exists.
\begin{prop}\label{prop 4.5}
	\begin{enumerate}
		\item If  $R_{\C}(\varphi,Z,Q,\lambda_0,N)<\infty$ for some $\lambda_0$, then  $R_{\C}(\varphi,Z,Q,\lambda,N)=0$ for all $\lambda>\lambda_0$.
		\item If $R_{\C}(\varphi,Z,Q,\lambda_0,N)=\infty$ for some $\lambda_0$, then  $R_{\C}(\varphi,Z,Q,\lambda,N)=\infty$ for all $\lambda<\lambda_0$.
	\end{enumerate}
\end{prop}
\begin{proof}
	It suffices to  show (1). Let $N\in \mathbb{N}$ and $\lambda >\lambda_0$.  Given   a  family   $\{Q_{n_i}(x_i,\C)\}_{i\in I}$ of cyclinder  sets such that $x_i \in Q$,  $n_i \geq N$ and  $\cup_{i\in I }Q_{n_i}(x_i,\C)\supset Z$,   we have 
	\begin{align*}
		&\sum_{i\in I}\limits   e^{-\lambda \M }\\
		=&\sum_{i\in I}\limits  e^{-\lambda_0\M+(\lambda_0-\lambda)\M}\\
		\leq&\sum_{i\in I}\limits   e^{-\lambda_0\M+(\lambda_0-\lambda)Nm\tau},
	\end{align*}
	where $m=\min_{u\in U}\varphi(u)>0$. This shows that  $R_{\C}(\varphi,Z,Q,\lambda,N)\leq R_{\C}(\varphi,Z,Q,\lambda_0)e^{(\lambda_0-\lambda)Nm\tau}$. Letting $N \to \infty$ gives us  the desired result.
\end{proof}
Proposition \ref{prop 4.5} tells us that  there is a critical value of the parameter $\lambda$  such  that  $R_{\C}(\varphi,Z,Q,\lambda)$ jumps from $\infty$ to 0.  

\begin{df}
The BS  invariance dimension   of $\varphi$ on $Z$ w.r.t. $\C$ is  defined  by 
\begin{align*}
{\rm dim}_{\C}^{BS}(\varphi,Z,Q):&=\inf\{\lambda: R_{\C}(\varphi,Z,Q,\lambda)=0\},\\
&=\sup\{\lambda:R_{\C}(\varphi,Z,Q,\lambda)=\infty\}.
\end{align*}
\end{df}

\begin{prop}
\begin{enumerate}
\item  If $\varphi =1$, then  ${\rm dim}_{\C}^{BS}(1,K,Q,)={\rm dim}_{\C}(K,Q)$.
\item If $Z_1\subset Z_2\subset Q$, then  $0\leq {\rm dim}_{\C}^{BS}(\varphi,Z_1,Q,\C)\leq {\rm dim}_{\C}^{BS}(\varphi,Z_2,Q,\C)$.
\item If $\lambda \geq 0$ and  $Z, Z_i\subset Q, i\geq1$  with $Z=\cup_{i\geq 1}Z_i$, then  $$R_{\C}(\varphi,Z,Q,\lambda)\leq \sum_{i\geq1}\limits R_{\C}(\varphi,Z_i,Q,\lambda)$$
and 
${\rm dim}_{\C}^{BS}(\varphi,Z,Q)=\sup_{i\geq 1}{\rm dim}_{\C}^{BS}(\varphi,Z_i,Q).$
\end{enumerate}
\end{prop}

We   are ready to give the proof of Theorem  \ref{thm 1.2}.
\begin{proof}[Proof of   Theorem \ref{thm 1.2}]
By definitions,  for each $N$ one has
$$R_{\C}(\varphi,Z,Q,\lambda,N)=M_{\C}(-\lambda\varphi,Z,Q,0,N)$$
and hence
$R_{\C}(\varphi,Z,Q,\lambda)=M_{\C}(-\lambda\varphi,Z,Q,0).$

Let $\lambda >{\rm dim}_{\C}^{BS}(\varphi,Z,Q)$. Then 
$M_{\C}(-\lambda\varphi,Z,Q,0)=R_{\C}(\varphi,Z,Q,\lambda)<1.$ So  $P_{\C}(-\lambda\varphi,Z,Q)\leq 0$.  By the proof of Proposition \ref{prop 3.3}, the continuity of $\Phi$ gives us 
$$P_{\C}(-{\rm dim}_{\C}^{BS}(\varphi,Z,Q)\cdot \varphi,Z,Q)\leq 0$$
by letting $\lambda  \to {\rm dim}_{\C}^{BS}(\varphi,Z,Q)$.  One  similarly  obtain that  $$P_{\C}(-{\rm dim}_{\C}^{BS}(\varphi,Z,Q)\cdot \varphi,Z,Q)\geq0.$$
Together with  Proposition \ref{prop 3.3},   ${\rm dim}_{\C}^{BS}(\varphi,Z,Q)$  is  the unique root of the equation $\Phi(t)=0$.
\end{proof}

The following corollary shows  the BS invariance dimension of $\psi$  on $Q$ w.r.t. $\C$ is a special case of $\psi$-induced invariance pressure of 0 on  $Q$  w.r.t. $\C$.

\begin{cor}
Let $\mathcal C =(\mathcal A, \tau,\nu)$ be an  invariant partition of  $Q$, and  $\psi\in C(U,\mathbb{R})$ with $\psi>0$.  If
the infimum in Eq. (\ref{equ 3.1})  is taken over all  finite families, then
$${\rm dim}_{\C}^{BS}(\psi,Q,Q)=P_{inv,\psi }(0,Q,\C).$$
\end{cor}

\begin{proof}
 If
the infimum in Eq. (\ref{equ 3.1})  is taken over all  finite families, then by \cite[Theorem 1]{zh19} $P_{\C}(f,Q,Q)=P_{inv}(f,Q,\C)$  for any $f\in C(U,\R)$. Taking $\varphi =0$ in  Theorem \ref{thm 1.1}, we deduce  that $$P_{inv}(-P_{inv,\psi }(0,Q,\C)\cdot \psi,Q,\mathcal{C})=P_{\C}(-P_{inv,\psi }(0,Q,\C)\cdot \psi,Q,Q)=0.$$ 
By Theorem \ref{thm 1.2},  we get  ${\rm dim}_{\C}^{BS}(\psi,Q,Q)=P_{inv,\psi }(0,Q,\C)$ by the  uniqueness of root of the equation. 
\end{proof}

\subsection{Variational principle for BS invariance dimension}
In this subsection,  we are devoted  to establishing a variational principle for BS invariance dimension.  The proof of Theorem \ref{thm 1.3}  is inspired by  the work of \cite{fh12,wc12,whs19,z20}.

By $M(Q)$ we denote the set of  Borel probability measures on $Q$.
\begin{df}
Let $Q$ be a  controlled  invariant set and  $\mathcal C =(\mathcal A, \tau,\nu)$ be an invariant partition of  $Q$, and let $\varphi\in C(U,\mathbb{R})$ with $\varphi >0$. 
Given $\mu \in M(Q)$, we  define  the measure-theoretic lower  BS invariance  pressure  of  $\mu$ w.r.t. $\varphi$  and $\C$ as
\begin{align*}
\underline{h}_{\mu,inv}(\varphi,Q,\C)&=\int_Q \liminf_{n \to \infty}-\frac{\log \mu(Q_n(x,\C))}{\S}d\mu.
\end{align*} 
\end{df}

\begin{rem}
\begin{enumerate}
\item   Given $n$,  put $$f_n(x):=\log \mu(Q_n(x,\C))$$ and $g_n(x):=\S>0$.  Then    the functions $f_n(x)$ and $g_n(x)$ are  Borel  measurable since they    only take finite many values.  This implies that $\liminf_{n \to \infty}-\frac{f_n(x)}{g_n(x)}$ is  also Borel measurable. 
\item  For  different  cylindrical set $Q_n(x,\C)$, the  weight is dynamically assigned to  a positive  number  $\S$ to detect  more information  of  the   decay rate of  cylindrical sets. Letting $\varphi=1$, it   generalizes   the concept of the measure-theoretic lower   invariance  entropy (cf. \cite[Subsection 4.1]{whs19})  by considering a  fixed  number $n\tau$. 
\end{enumerate}
\end{rem}

The  lower bound of  BS invariance dimension in Theorem \ref{thm 1.3}, i.e,  $$\sup\{\underline{h}_{\mu,inv}(\varphi,Q,\C):\mu \in M(Q), \mu (K)=1\}\leq {\rm dim}_{\C}^{BS}(\varphi,K,Q),$$
 is rather straightforward  by comparing their definitions.  
To obtain the upper bound of  ${\rm dim}_{\C}^{BS}(\varphi,K,Q)$, a  trick  learned from geometric measure theory \cite{m95} is \emph{defining  a  positive, bound linear functional  and then  using  the  Riesz representation theorem to produce a desired measure on $Q$}.

Let $Q$ be a  controlled  invariant set and  $\mathcal C =(\mathcal A, \tau,\nu)$ be an invariant partition of  $Q$,   let $\psi$ be a  bounded  function on $U$ and $\varphi\in C(U,\mathbb{R})$ with $\varphi >0$, $ \lambda\in \R, N\in \N$.  Define 
$$W_{\C}(\varphi,\psi,Q,\lambda,N)=\inf\left\{\sum_{i\in I}\limits c_i e^{-\lambda\M}\right\},$$
where the infimum  is taken over all  finite or countable families\\
$\{(Q_{n_i}(x_i,\C),c_i)\}_{i \in I}$ with  $0<c_i<\infty$, $x_i \in Q$, $n_i \geq N$ such that
$$\sum_{i\in I}\limits c_i \chi_{Q_{n_i}(x_i,\C)}\geq \psi,$$
where $\chi_{A}$ denotes the characteristic function of $A$. 

 The   definition of  $W_{\C}(\varphi,\psi,Q,\lambda,N)$ is highly inspired by the concept of   weighted Hausdorff  dimension  in classical  geometric measure theory \cite{m95} and \cite{fh12,whs19}. The   negative  bounded function $\psi$ is  not only  a counterpart of  the subset $Z$ given in BS invariance dimension by setting $\psi=\chi_Z$, but also   allows us  to   define  a  certain  positive, bounded linear functional, which is critical   for  obtaining a dynamical version of  \emph{Frostman's lemma} by invoking  some functional analysis techniques.

Let $Z\subset Q$ and put $W_{\C}(\varphi,Z,Q,\lambda,N):=W_{\C}(\varphi,\chi_{Z},Q,\lambda,N)$. We define 
$$W_{\C}(\varphi,Z,Q,\lambda)=\lim_{N\to \infty}W_{\C}(\varphi,Z,Q,\lambda,N).$$
Similar to  Proposition \ref{prop 4.5}, there is a critical value of $\lambda$ such  that  $W_{\C}(\varphi,Z,Q,\lambda)$ jumps from $\infty$ to $0$.

\begin{df}
The weighted BS  invariance dimension  on $Z$ w.r.t. $\C$ and $\varphi$ is defined by  
\begin{align*}
{\rm dim}_{\C}^{WBS}(\varphi,Z,Q):&=\inf\{\lambda: W_{\C}(\varphi,Z,Q,\lambda)=0\},\\
&=\sup\{\lambda:W_{\C}(\varphi,Z,Q,\lambda)=\infty\}.
\end{align*} 
\end{df}

\begin{prop}\label{prop 3.12}
Let  $\mathcal C =(\mathcal A, \tau,\nu)$ be an invariant partition of  $Q$ and  $Z\subset Q$ be a non-empty subset, and let $\varphi\in C(U,\mathbb{R})$ with $\varphi >0$.  Then for any $\lambda \geq 0$, $\epsilon >0$ and   sufficiently  large $N$, 
$$R_{\C}(\varphi,Z,Q,\lambda+\epsilon,N)\leq W_{\C}(\varphi,Z,Q,\lambda,N)\leq R_{\C}(\varphi,Z,Q,\lambda,N).$$
Consequently,  ${\rm dim}_{\C}^{BS}(\varphi,Z,Q)={\rm dim}_{\C}^{WBS}(\varphi,Z,Q).$
\end{prop}

\begin{proof}
It is clear that $W_{\C}(\varphi,Z,Q,\lambda,N)\leq R_{\C}(\varphi,Z,Q,\lambda,N)$ by definitions. Take $N_0\in N$  such  that 
$\frac{n^2}{e^{mn\epsilon\tau}}<1$ and $\sum_{k \geq n}\frac{1}{k^2}<1$ for all $n\geq N_0$,  where $m=\min_{u\in U}\varphi(u)>0$.   Fix $N\geq N_0$. We need to show for any  finite or countable family
$\{(Q_{n_i}(x_i,\C),c_i)\}_{i \in I}$ with  $0<c_i<\infty$, $x_i \in Q$, $n_i \geq N$ satisfying
$\sum_{i\in I}\limits c_i \chi_{Q_{n_i}(x_i,\C)}\geq \chi_{Z},$ one has
$$R_{\C}(\varphi,Z,Q,\lambda+\epsilon,N)\leq \sum_{i\in I}\limits c_i e^{-\lambda\M}.$$

Let $I_n=\{i\in I:n_i=n\}$. Without loss of generality, we  may assume that $Q_{n}(x_i,\C)\cap Q_{n}(x_j,\C)=\emptyset$  for all $i,j\in I_n$ with $i\not=j$. (Otherwise, one can   replace $(Q_{n}(x_i,\C),c_i)$ by $(Q_{n}(x_i,\C),c_i+c_j)$).  Let $s>0$ and  define
\begin{align*}
Z_{n,s}&=\{x\in Z:\sum _{i\in I_n}c_i\chi_{Q_{n}(x_i,\C)}(x)>s\}\\
I_{n,s}&=\{i\in I_n:Q_{n}(x_i,\C)\cap Z_{n,s}\not =\emptyset\}.
\end{align*}Then $Z_{n,s}\subset \cup_{i\in I_{n,s}}Q_{n}(x_i,\C)$. Therefore,
\begin{align*}
R_{\C}(\varphi,Z_{n,s},Q,\lambda+\epsilon,N)
&\leq \sum_{i\in I_{n,s}}e^{-(\lambda+\epsilon)S_{n\tau}\varphi(\omega_{\C_{[0,n)}(x_i)})}\\
&< \sum_{i\in I_{n,s}}\frac{c_i}{s}\cdot e^{-\lambda S_{n\tau}\varphi(\omega_{\C_{[0,n)}(x_i)})}\frac{1}{e^{\epsilon \cdot S_{n\tau}\varphi(\omega_{\C_{[0,n)}(x_i)})}}\\
&< \frac{1}{n^2s}\sum_{i\in I_{n}}c_i e^{-\lambda S_{n\tau}\varphi(\omega_{\C_{[0,n)}(x_i)})}.
\end{align*}
Let $s\in (0,1)$. If $x\in Z$, then we have $ \sum_{n\geq N}\limits \sum_{i\in I_n}\limits c_i\geq1> \sum_{n\geq N}\limits \frac{1}{n^2} s$. So there exists $n\geq N$ such that  $\sum_{i\in I_n}\limits c_i > \frac{1}{n^2} s$. This implies that  $x\in Z_{n,\frac{1}{n^2} s}$ and hence $Z= \cup_{n\geq N} Z_{n,\frac{1}{n^2} s}$. It follows that
\begin{align*}
R_{\C}(\varphi,Z,Q,\lambda+\epsilon,N)
&\leq \sum_{n\geq N}
R_{\C}(\varphi,Z_{n,\frac{1}{n^2} s},Q,\lambda+\epsilon,N)\\
&\leq \frac{1}{s}\sum_{n\geq N}\sum_{i\in I_{n}}c_i e^{-\lambda S_{n\tau}\varphi(\omega_{\C_{[0,n)}(x_i)})}\\
&= \frac{1}{s}\sum_{i\in I}\limits c_i e^{-\lambda\M}.
\end{align*}
We   finish the proof by  letting $s\to 1$.
\end{proof}

The following lemma is an analogue of  BS  Frostman's lemma for weighted BS invariance dimension.

\begin{lem}\label{prop 3.13}
Let   $\mathcal C =(\mathcal A, \tau,\nu)$ be a  clopen invariant partition of  $Q$. Let  $K$ be a non-empty compact subset of $Q$ and $ \lambda \geq0, N\in \mathbb{N}$, $\varphi\in C(U,\mathbb{R})$ with $\varphi >0$.  Suppose that $c:=W_{\C}(\varphi,K,Q,\lambda,N)>0$. Then there exists a  $\mu \in M(Q)$  with $\mu (K)=1$ such that
$$\mu(Q_n(x,\C))\leq \frac{1}{c}e^{-\lambda \S}$$
holds for all $x\in Q$ and $ n\geq N$.
\end{lem}
\begin{proof}
Clearly,  $c<\infty$. Define  a map $p$ on $C(Q,\mathbb{R})$  given by 
$$p(f)=\frac{1}{c}W_{\C}(\varphi,\chi_K\cdot f,Q,\lambda,N)$$
for any $f\in C(Q,\mathbb{R})$.  It is easy  to check that
\begin{enumerate}
\item $p(f+g)\leq p(f)+p(g)$ for any $f,g\in C(Q,\mathbb{R})$;
\item $p(kf)=kp(f)$ for any $k\geq 0$ and $f\in C(Q,\mathbb{R})$;
\item $p(\textbf{1})=1$, $p(f)\leq ||f||$ for any $f\in C(Q,\mathbb{R})$ and $p(g)=0$ for any $g\in C(Q,\mathbb{R})$ with $g\leq 0$, where $\textbf{1}$ denotes the constant function $\textbf{1}(1)=1$.
\end{enumerate} 

We only  show the inequality: $p(f)\leq ||f||$\footnote{Another simpler proof: $p(f)\leq p(||f||)=||f||\cdot p(1)=||f||$.}.
 Let
$\{(Q_{n_i}(x_i,\C),c_i)\}_{i \in I}$   be  a   finite or countable family with  $0<c_i<\infty$, $x_i \in Q$, $n_i \geq N$ such that
$\sum_{i\in I}\limits c_i \chi_{Q_{n_i}(x_i,\C)}\geq \chi_K$. Then  $\sum_{i\in I}\limits c_i||f|| \chi_{Q_{n_i}(x_i,\C)}\geq \chi_K f$
and hence $$p(f)\leq ||f||\cdot \frac{1}{c} \sum_{i\in I}\limits c_i e^{-\lambda\M}.$$  This shows that $p(f)\leq ||f||$.

By the Hahn-Banach theorem, we can extend the linear function $c\mapsto cp(1)$, $c\in \R$, from  the subspace of the constant functions to a linear functional $L:C(Q,\mathbb{R})\rightarrow \R$ satisfying
$L(1)=p(1)=1$ and $-p(-f)\leq L(f)\leq p(f)$ for any $f\in C(Q,\mathbb{R})$.
Let $f\in C(Q,\mathbb{R})$ with $f\geq0$. Then $p(-f)=0$ and hence $L(f)\geq 0$. By Riesz representation theorem, there exists  a $\mu \in M(Q)$ such that $L(f)=\int_Q f d\mu$ for any $f\in C(Q,\mathbb{R})$.

We show  $\mu(K)=1$.  For any compact  set $E\subset Q\backslash K$, by Urysohn's lemma there is a continuous  function $0\leq f\leq 1$  on $Q$ such that $f|_E=1$ and  $f|_K=0$. Then $\mu(E)\leq\int_Qfd\mu\leq p(f)=0$, which implies that $\mu(K)=1$.

We proceed to show  $\mu(Q_n(x,\C))\leq \frac{1}{c}e^{-\lambda \S}$
holds for all $x\in Q$ and $ n\geq N$. Let $n\geq N$ and   $x\in Q$.  For any compact set $E\subset Q_n(x,\C)$,  noting  that  $Q_n(x,\C)$ is open, the  Urysohn's lemma guarantees that there exists a continuous  function $0\leq f\leq 1$  on $Q$ such that $f|_E=1$ and  $f|_{Q\backslash Q_n(x,\C)}=0$. Then $\mu(E)\leq\int_Qfd\mu\leq p(f)$
and  $W_{\C}(\varphi,\chi_K\cdot f,Q,\lambda,N)\leq e^{-\lambda\S}$ since $\chi_{K}\cdot f\leq\chi_{Q_{n}(x_i,\C)}$. Hence, we get 
$$\mu(Q_n(x,\C))\leq \frac{1}{c}e^{-\lambda \S}.$$
\end{proof}

\begin{rem}
As we have mentioned,  $Q_n(x,\C)$ may not  be an  ``open ball". The assumption  imposed on  the  invariant partition   is to ensure    Urysohn's lemma  can be applied.
\end{rem}

Finally,  we give the proof of the Theorem \ref{thm 1.3}.
\begin{proof}[Proof of Theorem \ref{thm 1.3}]

We show $ \underline{h}_{\mu,inv}(\varphi,Q,\C)\leq {\rm dim}_{\C}^{BS}(\varphi,K,Q)$ for every  $\mu \in M(Q)$ with $\mu(K)=1$. Assume that  $\underline{h}_{\mu,inv}(\varphi,Q,\C)>0$ and let  $0<s<\underline{h}_{\mu,inv}(\varphi,Q,\C)$. Define
$$E=\{x\in Q:\liminf_{n \to \infty}-\frac{\log \mu(Q_n(x,\C))}{\S}>s\}.$$Then $\mu(E)>0$ and hence $\mu(E\cap K)>0$. We define
$$E_N=\{x\in E\cap K:-\frac{\log \mu(Q_n(x,\C))}{\S}> s ~\text{for any }n\geq N\}.$$
Then    $\mu(E_{N_0})>0$ for some $N_0\in \N$. Let $N\geq N_0$, and   let  $\{Q_{n_i}(x_i,\C)\}_{i\in I}$  be  a   finite or countable family  such that $x_i \in Q$,  $n_i \geq N$ and  $\cup_{i\geq I}Q_{n_i}(x_i,\C)\supset E_{N_0}$. We may assume that $Q_{n_i}(x_i,\C) \cap E_{N_0}\not =\emptyset$ for each $i \in I$.  Choose a point $y_i \in  Q_{n_i}(x_i,\C) \cap E_{N_0}$ for each $i \in I$. Then  $Q_{n_i}(x_i,\C)=Q_{n_i}(y_i,\C)$ and hence $\omega_{\C_{[0,n_i)}(x_i)}=\omega_{\C_{[0,n_i)}(y_i)}$.  Therefore, we have
\begin{align*}
\sum_{i\in I}e^{-s\M}&=\sum_{i\in I}e^{-sS_{n_i\tau}\varphi(\omega_{\C_{[0,n_i)}(y_i)})}\\
&\geq\sum_{i\in I}\mu(Q_{n_i}(y_i,\C)
\geq\mu( E_{N_0}).
\end{align*}
It follows that $R_{\C}(\varphi,E_{N_0},Q,s)\geq R_{\C}(\varphi,E_{N_0},Q,s,N)>0.$ This  implies that $${\rm dim}_{\C}^{BS}(\varphi,K,Q)\geq {\rm dim}_{\C}^{BS}(\varphi,E_{N_0},Q)\geq s.$$ Letting   $s\to \underline{h}_{\mu,inv}(\varphi,Q,\C)$, we get ${\rm dim}_{\C}^{BS}(\varphi,K,Q)\geq \underline{h}_{\mu,inv}(\varphi,Q,\C)$.

We continue to  show 
$${\rm dim}_{\C}^{BS}(\varphi,K,Q)\leq \sup\{\underline{h}_{\mu,inv}(\varphi,Q,\C):\mu \in M(Q), \mu (K)=1\}.$$ Let  $0<s< {\rm dim}_{\C}(\varphi,K,Q)$. By Proposition \ref{prop 3.12},  we have $W_{\C}(\varphi,K,Q,s)>0$ and hence there exists  $N$  such  that $c:=W_{\C}(\varphi,K,Q,s,N)>0$. 
By Lemma \ref{prop 3.13},  there exists   a  $\mu \in M(Q)$ such that $\mu (K)=1$ and $$\mu(Q_n(x,\C))\leq \frac{1}{c}e^{-s \S}$$
holds for all $x\in Q $ and $ n\geq N$. Then  $$s \leq \underline{h}_{\mu,inv}(\varphi,Q,\C)\leq  \sup\{\underline{h}_{\mu,inv}(\varphi,Q,\C):\mu \in M(Q), \mu (K)=1\}.$$  Letting $s \to {\rm dim}_{\C}(\varphi,K,Q)$, this completes the proof.
\end{proof}

\begin{rem}
Taking $\varphi =1$, we remark that   Theorem  \ref{thm 1.3} extends the previous variational principle for Bowen  invariance entropy  presented in \cite[Theorem 6.4]{whs19}.
\end{rem}

\section*{Acknowledgement} 

We sincerely thank the  editor 
H\'el\`ene Frankowska   and   the anonymous referees  for  abundant valuable  comments  and  insightful suggestions, which  greatly improved the  quality of the paper.

The first author was  supported by  the China Postdoctoral Science Foundation (No.2024M763856). The second author  was supported by the
National Natural Science Foundation of China (Nos. 12471184 and 12071222). The third  author was supported by Postgraduate Research $\&$ Practice Innovation Program of Jiangsu Province (No. KYCX24$\_$1790). The  fourth  author   was  supported by the
National Natural Science Foundation of China (No. 11971236) and Qinglan project. The work was also funded by the Priority Academic Program Development of Jiangsu Higher Education Institutions.  We would like to express our gratitude to Tianyuan Mathematical Center in Southwest China(No. 11826102), Sichuan University and Southwest Jiaotong University for their support and hospitality.


\end{document}